\documentclass[12pt]{article}%
\usepackage{graphicx,subcaption,caption,xcolor}
\usepackage{amsmath}
\usepackage{amsfonts}
\usepackage{amssymb}
\usepackage{float}
\usepackage{color}
\usepackage{multirow}
\usepackage{subeqnarray}
\usepackage{graphicx}%
\providecommand{\U}[1]{\protect\rule{.1in}{.1in}}
\providecommand{\U}[1]{\protect\rule{.1in}{.1in}}
\providecommand{\U}[1]{\protect\rule{.1in}{.1in}}
\providecommand{\U}[1]{\protect\rule{.1in}{.1in}}
\newtheorem{theorem}{Theorem}

\newtheorem{algorithm}[theorem]{Algorithm}

\numberwithin{equation}{section}
\numberwithin{theorem}{section}
\newtheorem{definition}[theorem]{Definition}

\newtheorem{lemma}[theorem]{Lemma}

\newtheorem{problem}[theorem]{Problem}
\newtheorem{proposition}[theorem]{Proposition}
\newtheorem{remark}[theorem]{Remark}

\newenvironment{proof}[1][Proof]{\textbf{#1.} }{\ \rule{0.5em}{0.5em}}

\begin{document}

\title{\textbf{The Implicit Convex Feasibility Problem and Its Application to
Adaptive Image Denoising\bigskip}}
\author{Yair Censor$^{1}$, Aviv Gibali$^{2}$, Frank Lenzen$^{3}$ and Christoph
Schn\"{o}rr$^{3}$\bigskip\\$^{1}$Department of Mathematics, University of Haifa,\\Mt.\ Carmel, 3498838 Haifa, Israel\smallskip\\$^{2}$Department of Mathematics, ORT Braude College,\\2161002 Karmiel, Israel\smallskip\\$^{3}$ Heidelberg Collaboratory for Image Processing,\\Mathematikon, INF 205, University of Heidelberg,\\69120 Heidelberg, Germany\smallskip}
\date{March 20, 2016. Revised: June 14, 2016.} \maketitle

\begin{abstract}
The implicit convex feasibility problem attempts to find a point
in the intersection of a finite family of convex sets, some of
which are not explicitly determined but may vary. We develop
simultaneous and sequential projection methods capable of handling
such problems and demonstrate their applicability to image
denoising in a specific medical imaging situation. By allowing the
variable sets to undergo scaling, shifting and rotation, this work
generalizes previous results wherein the implicit convex
feasibility problem was used for cooperative wireless sensor
network positioning where sets are balls and their centers were
implicit.\bigskip


\end{abstract}

\section{Introduction}

In this paper we are concerned with the following \textquotedblleft implicit
convex feasibility problem\textquotedblright\ (ICFP). Given set-valued
mappings $C_{s}:\mathbb{R}^{n}\rightarrow2^{\mathbb{R}^{n}}$, $s=1,2,\ldots,S,$\ with
closed and convex value sets, the ICFP is,%

\begin{equation}
\text{Find a point }x^{\ast}\in\cap_{s=1}^{S}C_{s}(x^{\ast}).
\end{equation}

We call the sets $C_{s}(x)$ \textquotedblleft variable sets\textquotedblright%
\ for obvious reasons and include \textquotedblleft implicit\textquotedblright%
\ in this problem name because the sets defining it are not given explicitly
ahead of time. The problem is inspired by the work of Gholami et al.
\cite{gtsc13} on solving the cooperative wireless sensor network positioning
problem in $\mathbb{R}^{2}$\ ($\mathbb{R}^{n}$). There, the sets $C_{s}(x)$\ are circles (balls)
with varying centers. A special instance of the ICFP is obtained by taking
fixed sets $C_{s}(x)\equiv C_{s},$\ for all $x\in \mathbb{R}^{n},$\ and all
$s=1,2,\ldots,S,$\ yielding the well-known, see, e.g., \cite{bb96},
\textquotedblleft convex feasibility problem\textquotedblright\ (CFP) which is,%

\begin{equation}
\text{Find a point }x^{\ast}\in\cap_{s=1}^{S}C_{s}.
\end{equation}

The CFP formalism is at the core of the modeling of many inverse
problems in various areas of mathematics and the physical
sciences. This problem has been widely explored and researched in
the last decades, see, e.g., \cite[Section 1.3]{Ceg12}, and many
iterative methods where proposed, in particular projection
methods, see, e.g., \cite{annotated}. These are iterative
algorithms that use projections onto sets, relying on the
principle that when a family of sets is present, then projections
onto the given individual sets are easier to perform than
projections onto other sets (intersections, image sets under some
transformation, etc.) that are derived from the given individual
sets.\\

Gholami et al. in \cite{gtsc13} introduced the implicit convex
feasibility problem (ICFP) in $\mathbb{R}^{d}$ ($d=2$ or $d=3$)
into their study of the wireless sensor network (WSN) positioning
problem. In their reformulation the variable sets are circles or
balls whose centers represent the sensors' locations and their
broadcasting range is represented as the radii. Some of these
centers are known a priori while the rest are unknown and need to
be determined. The WSN positioning problem is to find a point, in
an appropriate product space, which represents the circles or
balls centers. The precise relationship between the WSN problem
and the ICFP can be found in \cite[Section B]{gtsc13}. For more
details and other examples of geometric positioning problems, see
\cite{Gholami-thesis, gwsr11}.

We focus on the ICFP in $\mathbb{R}^{n}$\ and present projection
methods for its solution. This expands and generalizes the special
case treated in Gholami et al. \cite{gtsc13}. Moreover, we
demonstrate the applicability of our approach to the task of image
denoising, where we impose constraints on the image intensity at
every image pixel. Because the constraint sets depend on the
unknown variables to be determined, the method is able to adapt to
the image contents. This application demonstrates the usefulness
of the ICFP approach to image processing.

The paper is structured as follows. In Section \ref{sec:proj-implicit} we show
how to calculate projections onto variable sets. In Section \ref{sec:Algs} we
present two projection type algorithmic schemes for solving the ICFP,
sequential and simultaneous, along with their convergence proofs. In Section
\ref{sec:Special} we present the ICFP application to image denoising together
with numerical visualization of the performance of the methods. Finally, in
Section \ref{sec:summary} we discuss further research directions and propose a
further generalization of the ICFP.

\section{Projections onto variable convex sets\label{sec:proj-implicit}}

We begin by recalling the split
convex feasibility problem (SCFP) and the constrained multiple-set split
convex feasibility problem (CMSSCFP) that will be useful to our subsequent analysis.

\begin{problem}
\label{P:SCFP} Censor and Elfving \cite{ce94}. Given nonempty, closed and
convex sets $C\subseteq\mathbb{R}^{n},$ $Q\subseteq\mathbb{R} ^{m}$ and a
linear operator $T:\mathbb{R}^{n}\rightarrow\mathbb{R}^{m}$, the \texttt{Split
Convex Feasibility Problem} (SCFP) is:
\begin{equation}
\text{Find a point }x^{\ast}\in C\text{ such that }T(x^{\ast})\in Q.
\label{eq:scfp}%
\end{equation}

\end{problem}

Another related more general problem is the following.

\begin{problem}
\label{P:CMSSCFP} Masad and Reich \cite{mr07}. Let $r,p\in\mathbb{N}$ and
$\Omega_{s},$ $1\leq s\leq{S},$ and $Q_{r},$ $1\leq r\leq{R},$ be nonempty,
closed and convex subsets of $\mathbb{R}^{n}\ $and $\mathbb{R}^{m},$
respectively. Given linear operators $T_{r}:\mathbb{R}^{n}\rightarrow
\mathbb{R}^{m},$ $1\leq r\leq{R}$ and another nonempty, closed and convex
$\Gamma\subseteq\mathbb{R}^{n}$, the \texttt{Constrained Multiple-Set Split
Convex Feasibility Problem}\textit{\ }(CMSSCFP) is:
\begin{gather}
\text{Find a point }x^{\ast}\in\Gamma\text{ such that}\nonumber\\
x^{\ast}\in\cap_{s=1}^{S}\Omega_{s}\text{ and }T_{r}(x^{\ast})\in Q_{r}\text{
for each }r=1,2,\ldots,{R}. \label{eq:CMSSCFP-2}%
\end{gather}

\end{problem}

If $T_{r}\equiv T$ for all $r=1,2,\ldots,{R}$, then we obtain a multiple-set
split convex feasibility problem\textit{\ }(MSSCFP) \cite{cekb05}.

A prototype for the above SCFP and MSSCFP is the \textbf{Split
Inverse Problem }(SIP) presented in \cite{bcgr12,cgr12} and given
next.

\begin{problem} \label{P:SIP}
Given two vector spaces $X$ and $Y$ and a linear operator
$A:X\rightarrow Y$, we look at two inverse problems. One, denoted
by IP$_{1},$ is formulated in $X$ and the second, denoted by
IP$_{2}$, is formulated in $Y$. The \texttt{Split Inverse Problem}
(SIP) is:
\begin{equation}
\text{Find a point }x^{\ast}\in X\text{ that solves IP}_{1}
\text{such that } y^{\ast}=Ax^{\ast} \text{solves IP}_{2}.
\label{eq:sip}%
\end{equation}
\end{problem}

In \cite{bcgr12, cgr12} different choices for IP$_{1}$ and
IP$_{2}$ are proposed, such as variational inequalities and
minimization problems. The latter enable, for example, to obtain a
least-intensity feasible solution in intensity-modulated radiation
therapy (IMRT) treatment planning as in \cite{xcmg03}. In
\cite{gks14} we further explore and extend this modeling technique
to include non-linear mappings between the two spaces $X$ and $Y$.

Let $C\subseteq\mathbb{R}^{n}$ be a nonempty, closed and convex set. For each
point $x\in\mathbb{R}^{n}$,\ there exists a unique nearest point in $C$,
denoted by $P_{C}(x)$, i.e.,
\begin{equation}
\left\Vert x-P_{C}\left(  x\right)  \right\Vert \leq\left\Vert x-y\right\Vert
,\text{ for all }y\in C.
\end{equation}

The mapping $P_{C}:\mathbb{R}^{n}\rightarrow C$ is the \textit{metric projection} of
$\mathbb{R} ^{n}$ onto $C$. It is well-known that $P_{C}$ is a \textit{nonexpansive}
mapping of $\mathbb{R}^{n}$ onto $C$, i.e.,%

\begin{equation}
\left\Vert P_{C}\left(  x\right)  -P_{C}\left(  y\right)  \right\Vert
\leq\left\Vert x-y\right\Vert ,\ \text{for all }x,y\in\mathbb{R}^{n}.
\end{equation}

The metric projection $P_{C}$ is characterized by the following two
properties:%
\begin{equation}
P_{C}(x)\in C
\end{equation}
and%
\begin{equation}
\left\langle x-P_{C}\left(  x\right)  ,y-P_{C}\left(  x\right)  \right\rangle
\leq0,\text{ for all }x\in\mathbb{R} ^{n},\text{ }y\in C. \label{eq:ProjP1}%
\end{equation}

If $C$ is a hyperplane, then (\ref{eq:ProjP1}) becomes an equality,
\cite{GR84}. We are dealing with variable convex sets that can be described by
set-valued mappings.

\begin{definition}
\label{ex:1}For a set-valued mapping $C:\mathbb{R}^{n}\rightarrow
2^{\mathbb{R}^{n}},$ we call the sets $C(x)\subseteq\mathbb{R}^{n},$ defined
below, \textquotedblleft variable sets\textquotedblright. Let
$\Omega\subseteq\mathbb{R}^{n}$ be a given set, called in the sequel a
\textquotedblleft core set\textquotedblright.\medskip

(i) Given an operator $\mu:\mathbb{R}^{n}\rightarrow\mathbb{R}^{n}$, the
variable sets $C(x):=\Omega+\mu(x)=\left\{  y+\mu(x)\mid y\in\Omega\right\}
,$ for $x\in\mathbb{R}^{n}$, are obtained from shifting $\Omega$ by the
vectors $\mu(x)$.\medskip

(ii) Given an $x\in\mathbb{R}^{n}$ let $U_{\left[  x\right]  }:\mathbb{R}%
^{n}\rightarrow\mathbb{R}^{n}$ be a linear bounded operator. The variable sets
$C(x):=U_{\left[  x\right]  }(\Omega)=\left\{  U_{\left[  x\right]  }y\mid
y\in\Omega\right\}  $ are the $U_{\left[  x\right]  }$ images of $\Omega
$.\medskip

(iii) Given a function $f:\mathbb{R}^{n}\rightarrow\mathbb{R}_{+}$, the
variable sets $C(x):=f(x)\Omega=\left\{  f(x)y\mid y\in\Omega\right\}  ,$ for
$x\in\mathbb{R}^{n}$, are obtained from scaling $\Omega$ uniformly by $f(x)$.
This can be re-written as in (ii) with $U_{\left[  x\right]  }=f(x)I$, where
$I$ is the identity matrix.\medskip
\end{definition}

Next we present a lemma that shows how to calculate the metric projection onto
such variable sets via projections onto the core set $\Omega$ when the
operator $\mu$ is linear and denoted by the fixed matrix $A$ and $U_{\left[
x\right]  }$ is a constant unitary matrix denoted by $U$ (that is
$U^{T}U=UU^{T}=I$, where $U^{T}:\mathbb{R}^{n}\rightarrow\mathbb{R}^{n}$ is
the adjoint of $U$). Our proofs are based on Cegielski \cite[Subsection
1.2.3]{Ceg12}.

\begin{lemma}
\label{Lemma:1}Let $\Omega\subseteq\mathbb{R}^{n}$ be a nonempty, closed and
convex core set. Given a matrix $A$, a positive diagonal matrix $D=\alpha I,$
$\alpha>0,$ and a fixed unitary matrix $U$, the following holds for any
$z,x\in\mathbb{R}^{n}$%

\begin{equation}
P_{DU\left(  \Omega\right)  +Ax}\left(  z\right)  =DUP_{\Omega}\left(
D^{-1}U^{T}\left(  z-Ax\right)  \right)  +Ax.
\end{equation}

Since $D$ is a positive diagonal matrix, this can be re-written as%

\begin{equation}
P_{C(x)}\left(  z\right)  =\alpha UP_{\Omega}\left(  \frac{1}{\alpha}%
U^{T}\left(  z-Ax\right)  \right)  +Ax
\end{equation}

where $C(x)=\alpha U\left(  \Omega\right)  +Ax.$
\end{lemma}

\begin{proof}
Let $z\in\mathbb{R}^{n}$ and denote%
\begin{equation}
y:=\alpha UP_{\Omega}\left(  \frac{1}{\alpha}U^{T}\left(  z-Ax\right)
\right)  +Ax. \label{eq:y}%
\end{equation}

We show that%
\begin{equation}
y=P_{C(x)}\left(  z\right)  .
\end{equation}

From (\ref{eq:y}) and the unitary matrix $U$ we deduce%
\begin{equation}
\frac{1}{\alpha}U^{T}\left(  y-Ax\right)  =P_{\Omega}\left(  \frac{1}{\alpha
}U^{T}\left(  z-Ax\right)  \right)  .
\end{equation}

By the characterization of the metric projection onto $\Omega$
(\ref{eq:ProjP1}) we have%

\begin{equation}
\left\langle \frac{1}{\alpha}U^{T}\left(  z-Ax\right)  -\frac{1}{\alpha}%
U^{T}\left(  y-Ax\right)  ,w-\frac{1}{\alpha}U^{T}\left(  y-Ax\right)
\right\rangle \leq0,\text{ for all }w\in\Omega.
\end{equation}

Since $\alpha>0$ and $U$ is unitary, we get%

\begin{equation}
\left\langle z-y,\alpha Uw+Ax-y\right\rangle \leq0,\text{ for all }w\in\Omega.
\end{equation}

Denoting $v:=\alpha Uw+Ax$, since $w\in\Omega$ we get $v\in\alpha U\left(
\Omega\right)  +Ax=C(x)$,

the
\begin{equation}
\left\langle z-y,v-y\right\rangle \leq0,\text{ for all }v\in C(x),
\end{equation}

and again by the characterization of the metric projection onto
$C(x)$ (\ref{eq:ProjP1}) and by (\ref{eq:y}) $y=P_{C(x)}\left(
z\right) $ which completes the proof.
\end{proof}

\begin{figure}
[H]
\begin{center}
\includegraphics[
width=3in
]%
{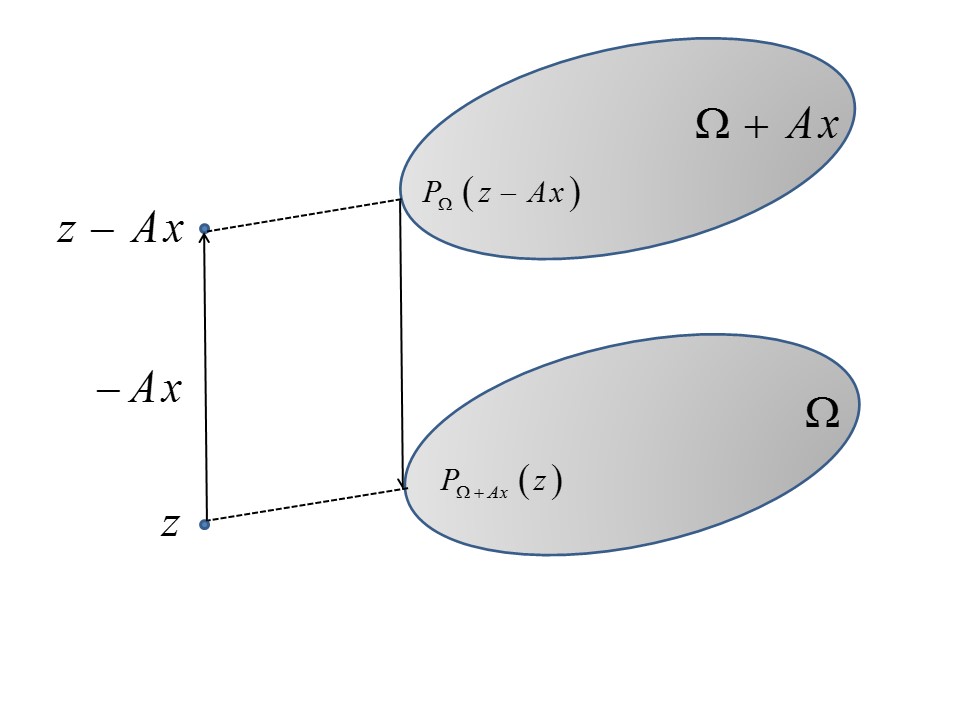}%
\caption{Illustration of Lemma \ref{Lemma:1} with $D,U=I$}%
\label{Fig:1a}%
\end{center}
\end{figure}
\begin{figure}
[H]
\begin{center}
\includegraphics[
width=3in
]%
{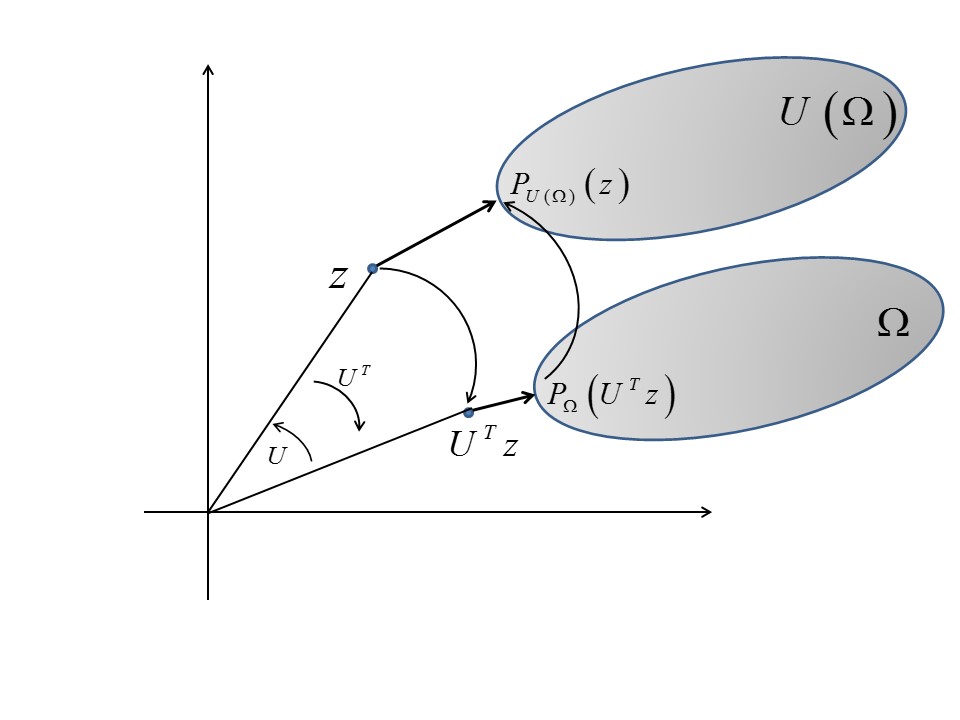}%
\caption{Illustration of Lemma \ref{Lemma:1} with $D=I$ and $A=0$}%
\label{Fig:2}%
\end{center}
\end{figure}

Two special cases of Lemma \ref{Lemma:1} are illustrated. In
Figure \ref{Fig:1a} we use $D,U=I$ so that $P_{\Omega+Ax}\left(
z\right) =P_{\Omega}\left(  z-Ax\right) +Ax$, meaning that the set
$\Omega$ is shifted by the point $Ax$ to $\Omega+Ax$. In Figure
\ref{Fig:2} we use $D=I$ and $A=0$ so that
$P_{U\left(\Omega\right)} \left( z\right)=UP_{\Omega }\left(
U^{T}z\right)$, meaning that the set $\Omega$ is rotated by the
unitary matrix $U$ to $U\left( \Omega\right)$.

\newpage
\section{The algorithms\label{sec:Algs}}

The following definitions will be used.

\begin{definition}
A sequence $\{\sigma_{k}\}_{k\in%
\mathbb{N}
}$ of real positive numbers is called a \texttt{steering sequence} if it
satisfies all the following conditions:\medskip\ (i) $\lim_{k\rightarrow
\infty}\sigma_{k}=0;$ (ii) $\lim_{k\rightarrow\infty}\frac{\displaystyle\sigma_{k+1}%
}{\displaystyle\sigma_{k}}=1;$ and (iii) $\sum_{k=0}^{\infty}\sigma_{k}=+\infty.$ Let
$\beta$ be a positive integer. If condition (ii) is replaced by%
\begin{equation}
\lim_{k\rightarrow\infty}\frac{\sigma_{k\beta+j}}{\sigma_{k\beta}%
}=1,\ \text{for all }1\leq j\leq\beta-1,
\end{equation}
and \ (i) and (iii) remain unchanged then the sequence is called an $\beta
$\texttt{-steering sequence.}
\end{definition}

\begin{definition}
A sequence $\{i(k)\}_{k\in%
\mathbb{N}
}$ of indices is called a \texttt{cyclic control sequence} over the index set
$\{1,2,\ldots,S\}$ if%
\begin{equation}
i(k)=k\operatorname{mod}S,\text{ \ for }k\geq0.
\end{equation}

\end{definition}

\begin{problem}
\label{P:ICFP}\textbf{The Implicit Convex Feasibility Problem}. Given
set-valued mappings $C_{s}:%
\mathbb{R}
^{n}\rightarrow2^{%
\mathbb{R}
^{n}}$, $s=1,2,\ldots,{S,}$ with closed convex values $C_{s}(x)$, the
\textit{Implicit Convex Feasibility Problem} (ICFP) is%
\begin{equation}
\text{Find a point }x^{\ast}\in\cap_{s=1}^{S}C_{s}(x^{\ast}).
\end{equation}

\end{problem}

One way of handling this problem is to reformulate it as the unconstrained
minimization%
\begin{equation}
\left\{
\begin{array}
[c]{ll}%
\mathrm{minimize} & G_{icfp}(x)\\
\mathrm{subject~to} & x\in\mathbb{R}^{n},
\end{array}
\right.  \label{p:unconstraint-minimization}%
\end{equation}
where%
\begin{equation}
G_{icfp}(x):=\frac{1}{2}\sum_{s=1}^{S}\Vert x-P_{C_{s}(x)}(x)\Vert^{2}
\label{func:G_icfp}%
\end{equation}
in which $P_{C_{s}(x)}$ is the metric projection operator onto the sets
$C_{s}(x)$.

Following the works of Censor et al. \cite{cdz04} and Gholami et al.
\cite{gtsc13} we present two algorithmic schemes, simultaneous and sequential,
for solving the ICFP of Problem \ref{P:ICFP}. For $s=1,2,\ldots,{S,}$ let
$\Omega_{s}$ be nonempty, closed and convex core sets in $\mathbb{R}^{n}$,
$A_{s}\in\mathbb{R}^{n\times n}$ are matrices$,$ $U_{s}\in\mathbb{R}^{n\times
n}$ are unitary matrices, and $\alpha_{s}>0$. Then the variable sets defined
in Definition \ref{ex:1} take the form%
\begin{equation}
C_{s}(x)=\alpha_{s}U_{s}\left(  \Omega_{s}\right)  +A_{s}x, \label{eq:Cs}%
\end{equation}
and we assume that the projection $P_{\Omega_{s}}$ are at hand or can be
easily calculated.

\begin{algorithm}
\label{alg:GC-sim}\textbf{The Simultaneous Algorithm}$\left.  {}\right.  $

\textbf{Preliminary calculations}: For $s=1,2,\ldots,{S,}$ calculate the
matrices%
\begin{equation}
K_{s}:=\frac{1}{\alpha_{s}}U_{s}^{T}\left(  I-A_{s}\right)  \label{eq:Ks}%
\end{equation}
and use matrix $2$-norms to calculate the constant%
\begin{equation}
L_{icfp}=\sum_{s=1}^{S}\left\Vert I-A_{s}\right\Vert _{2}^{2}.
\end{equation}

{\textbf{Initialization:}} Select an arbitrary starting point $x^{0}%
\in\mathbb{R}^{n}$ and set $k=0$.\newline

\textbf{Iterative step:} Given the current iterate $x^{k}$, calculate the next
iterate by%
\begin{equation}
x^{k+1}=x^{k}-\gamma_{k}\sum_{s=1}^{S}\alpha_{s}^{2}K_{s}^{T}\left(
I-P_{\Omega_{s}}\right)  \left(  K_{s}x^{k}\right)  , \label{eq:alg1:update_x}%
\end{equation}
where $\gamma_{k}\in(0,2/L_{icfp})$ for all $k\geq0.$\newline

\textbf{Stopping rule:} If\textbf{\ }$x^{k+1}=x^{k}$ (or, alternatively, if
$\Vert x^{k+1}-x^{k}\Vert$ is small enough) then stop. Otherwise, set
$k\leftarrow(k+1)$ and go back to the beginning of the iterative
step\textbf{.}
\end{algorithm}

\begin{algorithm}
\label{alg:GC-seq}\textbf{The Sequential Algorithm}$\left.  {}\right.  $

\textbf{Preliminary calculation}: For $s=1,2,\ldots,{S,}$ calculate the
matrices%
\begin{equation}
K_{s}:=\frac{1}{\alpha_{s}}U_{s}^{T}\left(  I-A_{s}\right)  .
\end{equation}

{\textbf{Initialization:}} Select an arbitrary starting point $x^{0}%
\in\mathbb{R}^{n}$ and set $k=0$.\newline

\textbf{Iterative step:} Given the current iterate $x^{k}$, calculate the next
iterate by%
\begin{equation}
x^{k+1}=x^{k}-\sigma_{k}\alpha_{i(k)}^{2}K_{i(k)}^{T}\left(  I-P_{\Omega
_{i(k)}}\right)  \left(  K_{i(k)}x^{k}\right)  ,
\end{equation}
where $\{\sigma_{k}\}_{k\in%
\mathbb{N}
}$ and $\{i(k)\}_{k\in%
\mathbb{N}
}$ are a $\beta$-steering and a cyclic control sequences,
respectively.\newline

\textbf{Stopping rule: }If\textbf{\ }$x^{k+1}=x^{k}$ (or, alternatively, if
$\Vert x^{k+1}-x^{k}\Vert$ is small enough) then stop. Otherwise, set
$k\leftarrow(k+1)$ and go back to the beginning of the iterative
step\textbf{.}
\end{algorithm}

\subsection{Convergence}

For the mappings $C_{s}(\cdot)$ of (\ref{eq:Cs}) we get, from Lemma
\ref{Lemma:1}, a simplified form of the proximity function $G_{icfp}$ in
(\ref{func:G_icfp}),%

\begin{align}
G_{icfp}(x) &  =\frac{1}{2}\sum_{s=1}^{S}\left\Vert x-\left(  \alpha_{s}%
U_{s}P_{\Omega_{s}}\left(  \frac{1}{\alpha_{s}}U_{s}^{T}\left(  x-A_{s}%
x\right)  \right)  +A_{s}x\right)  \right\Vert ^{2}\nonumber\\
&  =\frac{1}{2}\sum_{s=1}^{S}\left\Vert \alpha_{s}U_{s}\left(  I-P_{\Omega
_{s}}\right)  \left(  \frac{1}{\alpha_{s}}U_{s}^{T}\left(  I-A_{s}\right)
x\right)  \right\Vert ^{2}\nonumber\\
&  =\frac{1}{2}\sum_{s=1}^{S}\alpha_{s}^{2}\left\Vert \left(  I-P_{\Omega_{s}%
}\right)  \left(  \frac{1}{\alpha_{s}}U_{s}^{T}\left(  I-A_{s}\right)
x\right)  \right\Vert ^{2}\nonumber\\
&  =\frac{1}{2}\sum_{s=1}^{S}\alpha_{s}^{2}\left\Vert \left(  I-P_{\Omega_{s}%
}\right)  \left(  K_{s}x\right)  \right\Vert ^{2},\label{eq:Func_icfp}%
\end{align}
where $K_{s}$ is as in (\ref{eq:Ks}). \bigskip

In order to prove convergence of Algorithm \ref{alg:GC-sim} we need to show
that the function $G_{icfp}$ is convex, continuously differentiable and that
its gradient is Lipschitz continuous (see \cite[Proposition 4]{gtsc13}). For
the convergence of Algorithm \ref{alg:GC-seq} it is sufficient to show only
convexity and continuous differentiability of $G_{icfp}$, see \cite[Theorem
6]{cdz04}. In both cases our analysis relies on the classical theorems of
Baillon and Haddad \cite{bh77} and of Dolidze \cite{Dolidze84}.

\begin{proposition}
\label{prop1} The function $G_{icfp}$ of (\ref{eq:Func_icfp}) is (1) convex,
(2) continuously differentiable, and (3) its gradient is Lipschitz continuous.
\end{proposition}

\begin{proof}
Recall that the SCFP (\ref{P:SCFP}) can also be formulated as the minimization
problem%
\begin{equation}
\left\{
\begin{array}
[c]{ll}%
\mathrm{minimize} & G_{scfp}(x):=\frac{1}{2}\left\Vert Tx-P_{Q}\left(
Tx\right)  \right\Vert ^{2}\\
\mathrm{subject~to} & x\in C,
\end{array}
\right.
\end{equation}
see, e.g., \cite{byrne04}, and, moreover, for the CMSSCFP (\ref{eq:CMSSCFP-2})
we have%
\begin{equation}
\left\{
\begin{array}
[c]{ll}%
\mathrm{minimize} & G_{cmsscfp}(x):=\frac{1}{2}\sum_{s=1}^{S}\left\Vert
x-P_{\Omega_{s}}\left(  x\right)  \right\Vert ^{2}+\frac{1}{2}\sum_{r=1}%
^{R}\left\Vert T_{r}x-P_{Q_{r}}\left(  T_{r}x\right)  \right\Vert ^{2}\\
\mathrm{subject~to} & x\in\Gamma,
\end{array}
\right.
\end{equation}
see \cite{mr07}. Since our proximity function $G_{icfp}$ (\ref{eq:Func_icfp})
shares some common features with the above $G_{scfp}$ and $G_{cmsscfp}$
functions, we follow the lines of \cite{mr07} and \cite[Theorem 2]{gtsc13} to
prove the proposition.\medskip

(1) The convexity of $G_{icfp}$ (\ref{eq:Func_icfp}) is obvious, see, for
example \cite[Lemma 2]{mr07}.\medskip

(2) Since%
\begin{equation}
\nabla G_{scfp}(x)=T^{T}\left(  I-P_{Q}\right)  \left(  Tx\right)
\end{equation}
($T^{T}$ is the transpose of $T$) we deduce that%
\begin{equation}
\nabla G_{icfp}(x)=\sum_{s=1}^{S}\alpha_{s}^{2}K_{s}^{T}\left(  I-P_{\Omega
_{s}}\right)  \left(  K_{s}x\right)  ,
\end{equation}
where $K_{s}^{T}=\left(  1/\alpha_{s}\right)  \left(  I-A_{s}^{T}\right)
U_{s},$ and continuous differentiability follows.\medskip\

(3) To show that $\nabla G_{icfp}$ is Lipschitz continuous, that is%
\begin{equation}
\left\Vert \nabla G_{icfp}(x)-\nabla G_{icfp}(y)\right\Vert \leq
L_{icfp}\left\Vert x-y\right\Vert ,\text{ for all }x,y\in%
\mathbb{R}
^{n},
\end{equation}
we write%
\begin{align}
\nabla G_{icfp}(x)-\nabla G_{icfp}(y)  &  =\sum_{s=1}^{S}\alpha_{s}^{2}%
K_{s}^{T}\left(  I-P_{\Omega_{s}}\right)  \left(  K_{s}x\right)  -\sum
_{s=1}^{S}\alpha_{s}^{2}K_{s}^{T}\left(  I-P_{\Omega_{s}}\right)  \left(
K_{s}y\right) \nonumber\\
&  =\sum_{s=1}^{S}\alpha_{s}^{2}K_{s}^{T}\left(  I-P_{\Omega_{s}}\right)
\left(  K_{s}x-K_{s}y\right)  .
\end{align}
The firm-nonexpansivity of the projection operator, see, e.g., \cite[Definion
2.2.1]{Ceg12} along with the triangle and the Cauchy--Schwarz inequalities
imply%
\begin{align}
\left\Vert \nabla G_{icfp}(x)-\nabla G_{icfp}(y)\right\Vert  &  =\left\Vert
\sum_{s=1}^{S}\alpha_{s}^{2}K_{s}^{T}\left(  I-P_{\Omega_{s}}\right)  \left(
K_{s}x-K_{s}y\right)  \right\Vert \nonumber\\
&  \leq\sum_{s=1}^{S}\alpha_{s}^{2}\left\Vert K_{s}^{T}\right\Vert
_{2}\left\Vert K_{s}\right\Vert _{2}\left\Vert x-y\right\Vert \nonumber\\
&  =\sum_{s=1}^{S}\alpha_{s}^{2}\left\Vert K_{s}^{T}K_{s}\right\Vert
_{2}\left\Vert x-y\right\Vert .
\end{align}
Calculating%
\begin{align}
K_{s}^{T}K_{s}  &  =\frac{1}{\alpha_{s}^{2}}\left(  I-A_{s}^{T}\right)
U_{s}U_{s}^{T}\left(  I-A_{s}\right) \nonumber\\
&  =\frac{1}{\alpha_{s}^{2}}\left(  I-A_{s}^{T}\right)  \left(  I-A_{s}%
\right)  =\frac{1}{\alpha_{s}^{2}}\left(  I-A_{s}\right)  ^{T}\left(
I-A_{s}\right)
\end{align}
we obtain%
\begin{equation}
\left\Vert K_{s}^{T}K_{s}\right\Vert _{2}=\left(  \frac{1}{\alpha_{s}%
}\left\Vert I-A_{s}\right\Vert _{2}\right)  ^{2},
\end{equation}
meaning that%
\begin{equation}
\left\Vert \nabla G_{icfp}(x)-\nabla G_{icfp}(y)\right\Vert \leq\left(
\sum_{s=1}^{S}\left\Vert I-A_{s}\right\Vert _{2}^{2}\right)  \left\Vert
x-y\right\Vert,
\end{equation}
so that $\nabla G_{icfp}$ is Lipschitz continuous with the Lipschitz constant
$L_{icfp}=\sum_{s=1}^{S}\left\Vert I-A_{s}\right\Vert _{2}^{2}$.
\end{proof}

\begin{theorem}
For $s=1,2,\ldots,{S,}$ let $\Omega_{s}$ be nonempty, closed and convex core
sets in $\mathbb{R}^{n}$, $A_{s}\in\mathbb{R}^{n\times n}$ are matrices$,$
$U_{s}\in\mathbb{R}^{n\times n}$ are unitary matrices, and $\alpha_{s}>0$. If
the solution set of the ICFP of Problem \ref{P:ICFP} is nonempty then any
sequence $\left\{  x^{k}\right\}  _{k=0}^{\infty},$ generated by Algorithm
\ref{alg:GC-sim}$,$ converges to a solution $x^{\ast}$ of (\ref{P:ICFP}).
\end{theorem}

\begin{proof}
Proposition \ref{prop1}, guarantees that $G_{icfp}$ is convex, continuously
differentiable, and its gradient is Lipschitz continuous, therefore Algorithm
\ref{alg:GC-sim} is a gradient descent method for the unconstrained
minimization problem (\ref{p:unconstraint-minimization}) which solves the ICFP
(\ref{P:ICFP}). For the complete proof see, e.g., \cite[Proposition
2.3.2]{Bertsekas95}.
\end{proof}

\begin{theorem}
For $s=1,2,\ldots,{S,}$ let $\Omega_{s}$ be nonempty, closed and convex core
sets in $\mathbb{R}^{n}$, $A_{s}\in\mathbb{R}^{n\times n}$ are matrices$,$
$U_{s}\in\mathbb{R}^{n\times n}$ are unitary matrices, and $\alpha_{s}>0$. If
the solution set of the ICFP of Problem \ref{P:ICFP} is nonempty and any
sequence $\left\{  x^{k}\right\}  _{k=0}^{\infty}$ generated by Algorithm
\ref{alg:GC-seq} is bounded then $\left\{  x^{k}\right\}  _{k=0}^{\infty}%
$\textit{\ }converges to a solution $x^{\ast}$ of (\ref{P:ICFP}).
\end{theorem}

\begin{proof}
The function $G_{icfp}$ can be written as%
\begin{equation}
G_{icfp}(x)=\frac{1}{2}\sum_{s=1}^{S}g_{s}(x)
\end{equation}
where $g_{s}(x):=\alpha_{s}^{2}\left\Vert \left(  I-P_{\Omega_{s}}\right)
\left(  K_{s}x\right)  \right\Vert ^{2},$ for all $s=1,2,\ldots,{S}$. By
Proposition \ref{prop1} each $g_{s}$ is convex and continuously
differentiable, therefore Algorithm \ref{alg:GC-seq} is a special case of
\cite[Algorithm 5]{cdz04} and its convergence is guaranteed by \cite[Theorem
6]{cdz04}.
\end{proof}

\begin{remark}
1. Observe that the step size $\gamma_{k}$ in the simultaneous
projection algorithm \ref{alg:GC-sim} is chosen in the interval
$(0,2/L_{icfp})$ which requires the knowledge of the matrix
$2$-norm. As remarked by one of the referees this kind of step
size might be inefficient from the numerical point of view.
Several alternative step size strategies appear in \cite {sci15, lmwx12} and the references therein. \\

2. The ICFP of Problem \ref{P:ICFP} can be reformulated as an
unconstrained minimization so that by applying first-order methods
we get two different schemes that generate sequences that converge
to a solution of the ICFP of Problem \ref{P:ICFP}. An alternative
additional approach is to use the first-order optimality condition
in order to reformulate the ICFP of Problem \ref{P:ICFP} as a
variational inequality problem and derive other appropriate
algorithmic schemes, such as, Korpelevich's extragradient method
\cite{Korpelevich76}.
\end{remark}

\section{Application\label{sec:Special}}

\subsection{Model description}

In the following we introduce an approach for image denoising,
which is described in terms of an implicit convex feasibility
problem. We provide it as a specific instance of an ICFP rather
than as a method of choice for image denoising. Evaluating its
practical advantages for image denoising is a direction for future
work.

Various methods have been proposed in the literature for image
denoising. These methods can roughly be divided into methods based
on partial differential equations like the edge-preserving
Perona-Malik \cite{Catte1992} model and Weickert's anisotropic
diffusion \cite{Weickert1998}, Wavelet based methods
\cite{Chang2000}, non-local iterative filtering \cite{Buades2005},
collaborative filtering such as BM3D \cite{Dabov2007} and
variational approaches \cite{Scherzer2009variational}. Among the
latter we find methods based on regularization with the total
variation (TV) semi-norm \cite{Rudin1992} and higher order
expressions \cite{Bredies2010,Setzer2011}, which became popular
and widely used. Recent trends include also adaptive
\cite{Astroem2015,Estellers2015,Lefkimmiatis2015,Lenzen2015} and
non-local \cite{Kindermann2005} TV methods.

Below, we discuss an ansatz based on only prescribing constraints
for the pixel intensities, which leads to an ICFP. Since our
ansatz with fixed constraints (CFP) can be interpreted as a
constrained optimization problem with constant objective function,
it is related to the variational approaches. Allowing the
constraints to vary depending on the solution of the problem, such
as the ICFP allows, introduces adaptivity.

We now turn to the description of the proposed ICFP. For simplicity, we restrict ourselves to gray value images. We represent the
image to be denoised as a matrix ${Y=(y}_{i,j})\in\mathbb{R}^{{n_{1}}%
\times{n_{2}}}$, where ${n_{1}}$ and ${n_{2}}$ are the width and
height of the image. The noisy data $Y$ are obtained from an
unknown noise-free image represented by
${X^*}=(x^*_{i,j})\in\mathbb{R}^{{n_{1}}\times{n_{2}}}$ through
the relationship
\begin{equation}
y_{i,j}=x^*_{i,j} + \eta_{i,j},
\end{equation}
where $\eta_{i,j}\in\mathbb{R}^{{n_{1}}\times{n_{2}}}$ are
realizations of independent and identically distributed Gaussian
random variables with zero mean. We denote the denoised data $Y$
by $X{=(x}_{i,j})\in\mathbb{R}^{{n_{1}}\times{n_{2}}}$, which
forms our estimate of $X^{\ast}$.

For each pixel $(i,j)$ we will impose $S$ constraints on the gray level
intensity $x_{i,j}$ in terms of sets $\Omega_{s}^{i,j}$, $s=1,2,\dots,S$. Note
that we index an individual constraint set for an image location $(i,j)$ by a
subscript $s$, while the superscripts refer to the location. We motivate a
suitable choice for these sets as follows. Let us consider a fixed pixel
$(i,j)$ in the interior of the image together with its left and right
neighbors ${y}_{i-1,j}$ and ${y}_{i+1,j}$. In absence of noise, if the image
intensities vary smoothly, we can assume that ${y}_{i,j}$ is near the linear
interpolation of these two values, while in the case of strong noise
${y}_{i,j}$ likely lies outside the range determined by ${y}_{i-1,j}$ and
${y}_{i+1,j}$. Therefore, it makes sense to impose the constraint%
\begin{equation}\label{eq:denoising:horizontal_constraint}
x_{i,j}\in\Omega_{1}^{i,j}:=[\min({y}_{i+1,j},{y}_{i-1,j}),\max({y}%
_{i+1,j},{y}_{i-1,j})]
\end{equation}
for the smoothed image $X,$ where $[a,b]$ denotes the closed
interval between $a$ and $b$. To also cover the case of boundary
pixels, we assume a constant extension of the image outside the
image domain, so that \eqref{eq:denoising:horizontal_constraint}
is well-defined for every $(i,j)$.

We remark that there is a relation to TV regularization, since the
total variation of
the discrete signal $(\min({y}_{i+1,j},{y}_{i-1,j}),x_{i,j},\max({y}%
_{i+1,j},{y}_{i-1,j}))$ is minimal for $x_{i,j}\in\Omega_{1}^{i,j}$.

Analogously to \eqref{eq:denoising:horizontal_constraint} we
define constraint sets for every horizontal, vertical and diagonal
edge of the underlying grid graph with vertices corresponding to
the pixel positions $(i,j)$:
\begin{equation}
\begin{aligned} \Omega_2^{i,j} &:=[\min({y}_{i,j+1},{y}_{i,j-1}),\max({y}_{i,j+1},{y}_{i,j+1})],\\ \Omega_3^{i,j} &:=[\min({y}_{i+1,j+1},{y}_{i-1,j-1}),\max({y}_{i+1,j+1},{y}_{i-1,j-1})],\\ \Omega_4^{i,j} &:=[\min({y}_{i+1,j-1},{y}_{i-1,j+1}),\max({y}_{i+1,j-1},{y}_{i-1,j+1})].\\ \end{aligned} \label{eq:denoising:constraint_sets}%
\end{equation}
In total, we end up with four different constraint sets for pixel $(i,j)$.


Note that we can express each set $\Omega_{s}^{i,j}$ in the form%
\begin{equation}
\Omega_{s}^{i,j}=[-r_{s}^{i,j}({Y}),r_{s}^{i,j}({Y})]+m_{s}^{i,j}({Y}),
\end{equation}
where
\begin{equation}
\begin{aligned} r_1^{i,j}({Y}) &=\frac{1}{2}|{y}_{i+1,j}-{y}_{i-1,j}|, &&& m_1^{i,j}({Y}) &=\frac{1}{2}({y}_{i+1,j}+{y}_{i-1,j}), \end{aligned}
\end{equation}
and $r_{s}^{i,j}({Y})$, and $m_{s}^{i,j}({Y})$, $s=2,3,4$, defined accordingly for the vertical and the two diagonal directions.

We further slightly generalize the sets $\Omega_{s}^{i,j}$ by introducing a
scaling factor $\alpha>0$ and re-define%
\begin{equation}
\Omega_{s}^{i,j}:=\alpha\,[-r_{s}^{i,j}({Y}),r_{s}^{i,j}({Y})]+m_{s}^{i,j}%
({Y}),\text{ for }s=1,2,3,4.
\end{equation}
Based on these local constraint sets, we look at the CFP%
\begin{equation}
\text{Find }X{=(x}_{i,j})\in\mathbb{R}^{{n_{1}}\times{n_{2}}}\text{ such that
}x_{i,j}\in\bigcap_{s=1}^{4}\Omega_{s}^{i,j}\text{ for all pixels }(i,j).
\label{eq:denoising:cfp}%
\end{equation}
Recall that our approach presented in Section~\ref{sec:Algs} allows
$m_{s}^{i,j}$ to depend on $X$, in contrast with \eqref{eq:denoising:cfp}. So,
we consider the set-valued mappings
\begin{equation}
C_{s}^{i,j}(X):=\alpha\,[-r_{s}^{i,j}({Y}),r_{s}^{i,j}({Y})]+m_{s}^{i,j}(X),
\label{eq:denoising:scaled_Cij_x}%
\end{equation}
and derive the ICFP
\begin{equation}
\text{Find }X\in\mathbb{R}^{{n_{1}}\times{n_{2}}}\text{ such that }X\in%
{\displaystyle\prod\limits_{i,j}}
\left(  \bigcap_{s=1}^{4}C_{s}^{i,j}(X)\right)  , \label{eq:denoising:icfp}%
\end{equation}
where $\prod\limits_{i,j}$ represents the product of sets. Note that the variable sets $C_{s}^{i,j}(X)$
attain the form \eqref{eq:Cs}.

\subsection{Experiments}

In our computational experiments we consider two test images. The Shepp-Logan
phantom of Figure~\ref{fig:denoising:CFPversusICFP}(a), displayed with
Gaussian noise of zero mean and variance $0.1$ in
Figure~\ref{fig:denoising:CFPversusICFP}(b), and an ultrasound image,
displayed in Figure~\ref{fig:denoising:CFPversusICFP}(c).

For the ICFP~\eqref{eq:denoising:icfp}, we implemented both
Algorithms~\ref{alg:GC-sim} and Algorithm~\ref{alg:GC-seq}. The
parameters were chosen to be $\gamma_k=\frac{1}{16}$ and $1000$
iteration steps for Algorithm~\ref{alg:GC-sim} and $\sigma_{\beta
k+j}:=\displaystyle{\frac{1}{k}}$ for $j=0,1,\dots,\beta$ and
$1000$ iteration steps for Algorithm~\ref{alg:GC-seq}. If not
noted otherwise, we use $\beta=100$ for the latter.

We compared the performance of the CFP \eqref{eq:denoising:cfp}
with that of the ICFP \eqref{eq:denoising:icfp}. For both problems
we focused on the simultaneous projection method of
Algorithm~\ref{alg:GC-sim}, which is known to converge even in the
inconsistent case (i.e., the case where the intersection of
constraint sets is empty), see, e.g, \cite{Ceg12,popa12}.

In Figure~\ref{fig:denoising:CFPversusICFP} we present results of comparing
the CFP and ICFP approaches on the noisy Shepp-Logan test image and on the
ultrasound image that we used. We also provide close-ups for a specific region
of interest, in order to highlight the differences mainly in texture. We
observe that the CFP is not suited to denoise the data. In contrast, solving
the ICFP leads to a denoised image.

In Figure~\ref{fig:denoising:varying_a} we study the influence of the
parameter $\alpha$ in \eqref{eq:denoising:scaled_Cij_x} for the ICFP. Our
experiments show, that this parameter influences the smoothness of the result.
The smaller $\alpha$ the smoother the result becomes.

In Figure~\ref{fig:denoising:empty_sets}, we plot the percentage of constraint sets $\cap_{s=1}^{4}C_{s}^{i,j}(x^{k})$, $i=1,2,\dots,{n_{1}},$ $j=1,2,\dots,{n_{2},}$ depending on the iteration index $k$ for CFP and
ICFP, using both algorithms for the latter.

Note that in the CFP the constraint sets do not vary and, therefore, we have a constant fraction of
empty intersections, while in the ICFP we observe that the number of
empty intersections decreases significantly to a final percentage of 3.5\%,
showing that the constraint sets adapt to the unknown in a meaningful way.

We note that Algorithm~\ref{alg:GC-seq} requires a
$\beta$-steering sequence for convergence. To demonstrate the
influence of this sequence on the speed of convergence we conduct
an experiment with different values of $\beta$. Obviously the
choice of $\beta$ influences the solution, to which the iterative
sequence generated by the algorithm converges. We found that the
results are of similar quality (the SSIM index proposed by Wang et
al.~\cite{Wang2004} varies only in the range $0.6802\pm 0.0001$).
Due to the non-uniqueness of the solution, we can measure
convergence only with respect to the individual solution the
algorithm converges to. To this end, we assume that the sequence
converges within the first 1000 steps. This assumption is
satisfied, since the difference $\|X^{k-1}-X^k\|$ becomes small
for $k\approx 1000$. The plot in Figure~\ref{fig:denoising:betas}
shows the distances $d_k:=\|X^k-X^{1000}\|$ for
$\beta=10,20,50,100$ during the first $1000$ steps. We observe
that the larger $\beta$ is, the faster $d_k$ decreases. Thus, for
a faster convergence a larger value of $\beta$ is advantageous.

We conclude that the proposed ICFP has the capability of denoising image data.
Although the approach in its current form cannot cope with complex
state-of-the-art denoising approaches, our experiments demonstrate the
usefulness of imposing constraints on image intensities. Moreover, we see the
potential for further improvements of the approach, for example by
additionally making the parameter $\alpha$ depend on the unknown, or by
combining the ICFP with an objective function for denoising, that has to be
optimized subject to the given adaptive constraints.

\newcommand{\figsize}{3.5cm}%
\begin{figure}[H]%
\captionsetup[subfigure]{justification=centering}\centering%
\begin{tabular}
[c]{cccc}%
\begin{subfigure}[t]{3.5cm}\includegraphics[width=3.5cm]{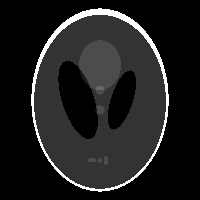}\subcaption{}\end{subfigure} &
\begin{subfigure}[t]{3.5cm}\includegraphics[width=3.5cm]{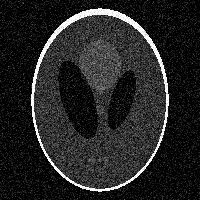}\subcaption{}\end{subfigure} &
\multicolumn{2}{c}{\begin{subfigure}[t]{\figsize}\includegraphics[width=3.5cm]{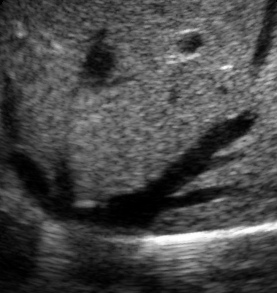}\subcaption{}\end{subfigure}}%
\\
\begin{subfigure}[t]{3.5cm}\includegraphics[width=3.5cm]{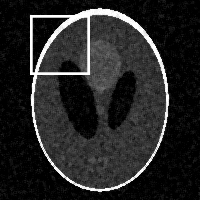}\subcaption{}\end{subfigure} &
\begin{subfigure}[t]{3.5cm}\includegraphics[width=3.5cm]{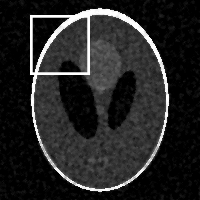}\subcaption{}\end{subfigure} &
\begin{subfigure}[t]{3.5cm}\includegraphics[width=3.5cm]{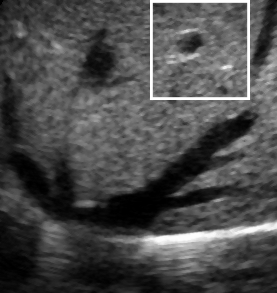}\subcaption{} \end{subfigure} &
\begin{subfigure}[t]{3.5cm}\includegraphics[width=3.5cm]{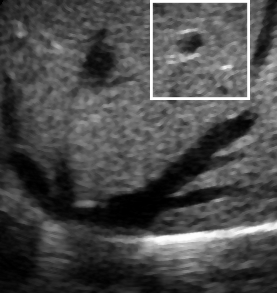}\subcaption{}\end{subfigure}\\
\begin{subfigure}[t]{3.5cm}\includegraphics[width=3.5cm]{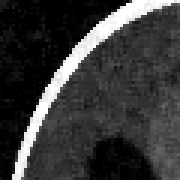}\subcaption{}\end{subfigure} &
\begin{subfigure}[t]{3.5cm}\includegraphics[width=3.5cm]{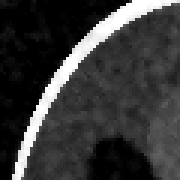}\subcaption{}\end{subfigure} &
\begin{subfigure}[t]{3.5cm}\includegraphics[width=3.5cm]{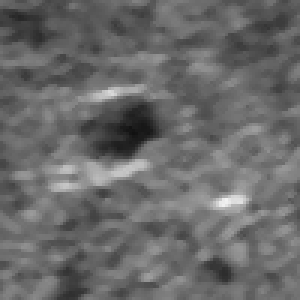}\subcaption{}\end{subfigure} &
\begin{subfigure}[t]{3.5cm}\includegraphics[width=3.5cm]{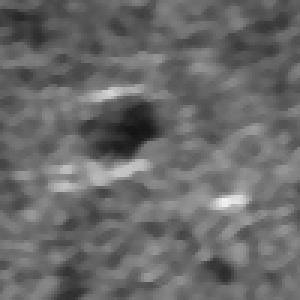}\subcaption{}\end{subfigure}
\end{tabular}
\caption{Comparison of the results of the CFP \eqref{eq:denoising:cfp} (d),
(f) and the ICFP \eqref{eq:denoising:icfp} (e), (g) and close-ups (h)-(k) of
the regions marked in white. From the close-ups, we observe that the ICFP is
better suited as a denoising method than the CFP. }%
\label{fig:denoising:CFPversusICFP}%
\end{figure}

\begin{figure}[H]
\captionsetup[subfigure]{justification=centering} \centering
\begin{subfigure}[t]{.35\linewidth}
\includegraphics[width=\linewidth]{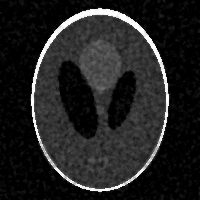}\
\includegraphics[width=\linewidth]{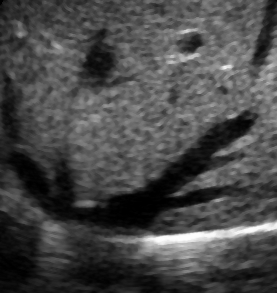}\\
\subcaption{{$\alpha=1$}}
\end{subfigure}
\begin{subfigure}[t]{.35\linewidth}
\includegraphics[width=\linewidth]{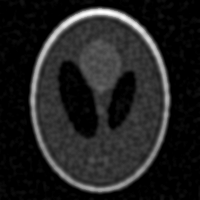}\
\includegraphics[width=\linewidth]{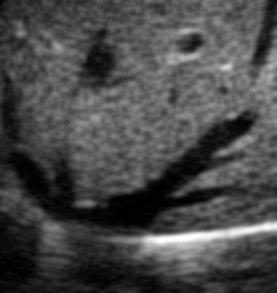}\\
\subcaption{{$\alpha=0.1$}}
\end{subfigure}
\caption{Results for varying parameter $\alpha$ (cf.
\eqref{eq:denoising:scaled_Cij_x}). Observe that decreasing $\alpha$ leads
to a smoother image. }%
\label{fig:denoising:varying_a}%
\end{figure}

\begin{figure}[H]
\begin{center}
\includegraphics[width=.8\linewidth]{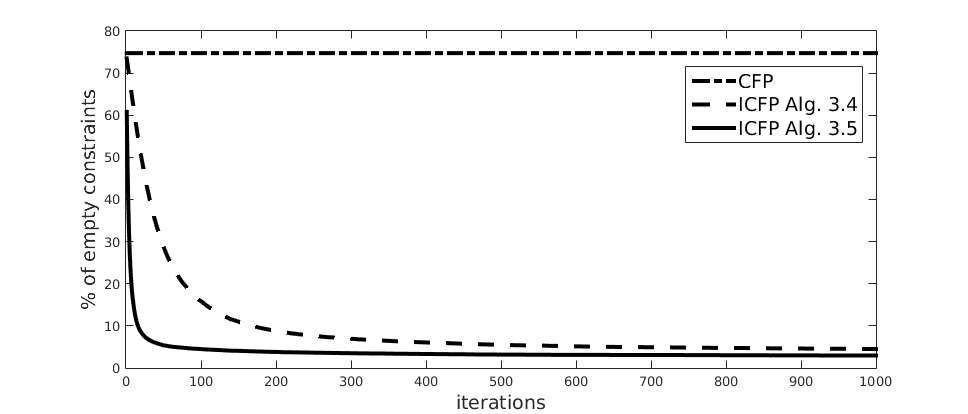}
\end{center}
\caption{Percentage of empty constraint sets $\cap_{s=1}^{4}C_{s}^{i,j}%
(X)$, $i=1,2,\dots,{n_{1}}$, $j=1,2,\dots,{n_{2}}$ (inconsistent
cases) with $x$ varying during the iterations. We compare the
algorithm for the CFP (dash-dotted), and for ICFP
Algorithm~\ref{alg:GC-sim} (dashed) and the sequential
Algorithm~\ref{alg:GC-seq} (solid). We observe, that for the ICFP,
the percentage significantly decreases during the iterations,
while for the CFP by
definition it is constant. }%
\label{fig:denoising:empty_sets}%
\end{figure}

\begin{figure}[H]
\begin{center}
\includegraphics[width=.8\linewidth]{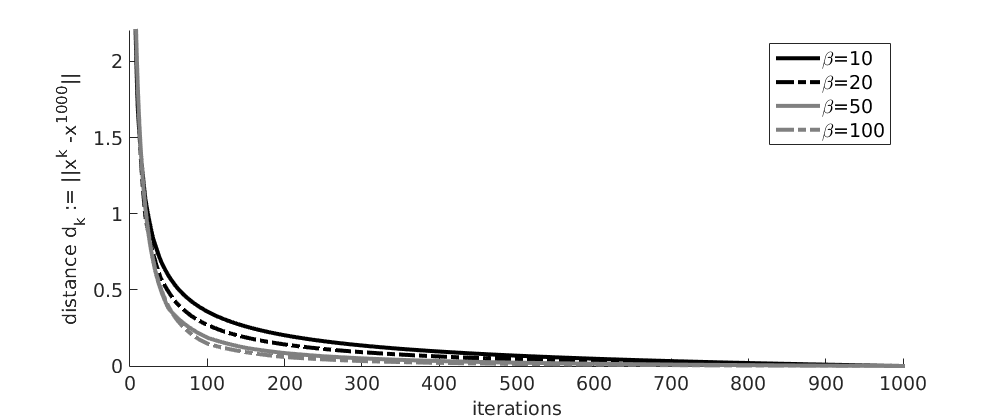}
\end{center}
\caption{ Convergence for different $\beta$-steering sequences
with $\beta=10$ (black solid), $20$ (black dashed), $50$ (gray
solid) and  $100$ (gray dashed) in Algorithm~\ref{alg:GC-seq}
applied for smoothing the phantom image. Note that the limit
depends on the chosen sequence. We depict $\|X^k-X^{1000}\|$,
assuming convergence within the first 1000 steps.
We conclude that a larger $\beta$ is advantageous for a faster
convergence.}%
\label{fig:denoising:betas}%
\end{figure}

\section{Summary and further discussion\label{sec:summary}}

In this paper we consider the implicit convex feasibility problem (ICFP) where
the variable sets are obtained by shifting, rotating and linearly-scaling
fixed, closed convex sets. By reformulating the problem as an unconstrained
minimization we present two algorithmic schemes for solving the problem, one
simultaneous and one sequential. We also comment that other first-order
methods can be applied if, for example, the problem is phrased as a
variational inequality problem. We illustrate the usefulness of the ICFP as a
new modeling technique for imposing constraints on image intensities in image denoising.

Two instances of the ICFP, the wireless sensor network (WSN) positioning
problem and the new image denoising approach suggest the applicability
potential of the ICFP. In this direction we recall the nonlinear multiple-sets
split feasibility problem (NMSSFP) introduced by Li et al. \cite{lhz} and
later by Gibali et al. \cite{gks14}.\newline

In this problem the linear operator $T:%
\mathbb{R}
^{n}\rightarrow%
\mathbb{R}
^{m}$ in the split convex feasibility problem (\ref{eq:scfp}) is nonlinear
and, therefore, the corresponding proximity function is not necessarily convex
which means that additional assumptions on $T$ are required, such as
differentiability. Within this framework it will be interesting to know, for
example, what are the necessary assumptions on $m:%
\mathbb{R}
^{n}\rightarrow%
\mathbb{R}
^{n}$ in Definition \ref{ex:1} which will guarantee convergence of our
proposed schemes.

Another direction is when the unitary matrices $U_{s}$ are not given in
advance but generated via some procedure; for example, given a linear
transformation $M:%
\mathbb{R}
^{n}\rightarrow%
\mathbb{R}
^{n\times n}$, such that for all $x\in%
\mathbb{R}
^{n}$, $M(x)=U_{[x]}$ is a unitary matrix. The linearity assumption on $M$
will guarantee that our analysis here will still hold true. For a nonlinear
$M$ our present analysis will not hold or, at least, not directly
hold.\bigskip

\textbf{Acknowledgments}. We thank the anonymous referees for
their comments and suggestions which helped us improve the paper.
The first author's work was supported by Research Grant No.
2013003 of the United States-Israel Binational Science Foundation
(BSF).


\begin{thebibliography}{99}                                                                                               %

\bibitem{Astroem2015}
F.~Astr{\"o}m, G.~Baravdish, and M.~Felsberg,
\newblock A tensor variational formulation of gradient energy total variation.
\newblock In X.-C. Tai, E.~Bae, T.~Chan, and M.~Lysaker, editors, {\em Energy
  Minimization Methods in Computer Vision and Pattern Recognition}, volume 8932
  of {\em Lecture Notes in Computer Science}, pages 307--320. Springer, 2015.


\bibitem {bh77}J. Baillon and G. Haddad, Quelques propri\'{e}t\'{e}s des
op\'{e}rateurs angelborn\'{e}s et n-cycliquement monotones, \textit{Israel
Journal of Mathematics} \textbf{26} (1977), 137--150.

\bibitem {bb96}H. H. Bauschke and J. M. Borwein, On projection algorithms for
solving convex feasibility problems,\ \textit{SIAM Review} \textbf{38} (1996), 367--426.

\bibitem {Bertsekas95}D. P. Bertsekas, \textit{Nonlinear Programming},
Belmont, MA, USA: Athena Scientific, 1995.

\bibitem{Bredies2010}
K.~Bredies, K.~Kunisch, and T.~Pock,
\newblock {T}otal {G}eneralized {V}ariation,
\newblock {\em SIAM Journal on Imaging Sciences} \textbf{3}(2010), 492--526.

\bibitem{Buades2005}
A.~Buades, B.~Coll, and J.~M. Morel,
\newblock A review of image denoising algorithms, with a new one,
\newblock {\em SIAM Journal on Multiscale Modeling and Simulation} \textbf{4} (2005), 490--530.

\bibitem {byrne04}C. Byrne, A unified treatment of some iterative algorithms
in signal processing and image reconstruction, \textit{Inverse Problems
}\textbf{20} (2004), 103--120.

\bibitem {bcgr12}C. Byrne, Y. Censor, A. Gibali and S. Reich, The split common
null point problem, \textit{Journal of Nonlinear and Convex
Analysis} \textbf{13} (2012), 759--775.

\bibitem{Catte1992}
F.~Catt\'e, P.-L. Lions, J.-M. Morel, and T.~Coll,
\newblock Image selective smoothing and edge detection by nonlinear
diffusion, \newblock {\em SIAM Journal on Numerical Analysis}
\textbf{29} (1992), 182--193.

\bibitem {Ceg12}A. Cegielski, \textit{Iterative Methods for Fixed Point
Problems in Hilbert Spaces}, Springer, Heidelberg, 2012.

\bibitem {annotated}Y. Censor and A. Cegielski, Projection methods: an
annotated bibliography of books and reviews, \textit{Optimization} \textbf{64
}(2015), 2343--2358.

\bibitem {cdz04}Y. Censor, A. R. De Pierro and M. Zaknoon, Steered sequential
projections for the inconsistent convex feasibility problem, \textit{Nonlinear
Analysis: Theory, Methods and Applications} \textbf{59} (2004), 385--405.

\bibitem {ce94}Y. Censor and T. Elfving , A multiprojection algorithm using
Bregman projections in a product space, \textit{Numerical Algorithms
}\textbf{8} (1994), 221--239.

\bibitem {cekb05}Y. Censor, T. Elfving, N. Kopf and T. Bortfeld, The
multiple-sets split feasibility problem and its applications for inverse
problems, \textit{Inverse Problems} \textbf{21 }(2005), 2071--2084.


\bibitem {cgr12}Y. Censor, A. Gibali and S. Reich, Algorithms for the split
variational inequality problem, \textit{Numerical Algorithms}
\textbf{59} (2012), 301--323.

\bibitem{Chang2000}
S.~G. Chang, B.~Yu, and M.~Vetterli,
\newblock Adaptive wavelet thresholding for image denoising and
compression, \newblock {\em IEEE Transactions on Image Processing}
\textbf{9} (2000), 1532--1546.

\bibitem{Dabov2007}
K.~Dabov, A.~Foi, V.~Katkovnik, and K.~Egiazarian,
\newblock Image denoising by sparse 3-d transform-domain collaborative
filtering, \newblock {\em IEEE Transactions on Image Processing},
\textbf{16} (2007), 2080--2095.


\bibitem {Dolidze84}Z. O. Dolidze, On convergence of an analog of the gradient
method, \textit{Ekonomika i Matematicheskie Metody} \textbf{20} (1984), 755--758.


\bibitem{Estellers2015}
V.~Estellers, S.~Soato, and X.~Bresson,
\newblock Adaptive regularization with the structure tensor,
\newblock {\em IEEE Transactions on Image Processing}, \textbf{24} (2015), 1777--1790.

\bibitem {Gholami-thesis}M. R. Gholami, \textit{Positioning Algorithms for
Wireless Sensor Networks}, Licentiate Thesis 2011, Chalmers
University of Technology, Gothenburg, Sweden.

\bibitem {gtsc13}M. R. Gholami, L. Tetruashvili, E. G. Str\"{o}m and Y. Censor,
Cooperative wireless sensor network positioning via implicit
convex feasibility, \textit{IEEE Transactions on Signal
Processing} \textbf{61} (2013), 5830--5840.

\bibitem {gwsr11}M. R. Gholami and H. Wymeersch and E. G. Str\"{o}m and M.
Rydstr\"{o}m, Wireless network positioning as a convex feasibility
problem, \textit{EURASIP Journal on Wireless Communications and
Networking} \textbf{161} (2011), 1--15.

\bibitem {gks14}A. Gibali, K.-H. K\"{u}fer and P. S\"{u}ss, Successive linear
programing approach for solving the nonlinear split feasibility problem,
\textit{Journal of Nonlinear and Convex Analysis} \textbf{15} (2014), 345--353.

\bibitem {GR84}K. Goebel and S. Reich, \textit{Uniform Convexity, Hyperbolic
Geometry, and Nonexpansive Mappings}, Marcel Dekker, New York and Basel, 1984.




\bibitem{Kindermann2005}
S.~Kindermann, S.~Osher, and P.~Jones,
\newblock Deblurring and denoising of images by nonlocal
functionals, \newblock {\em SIAM-Multiscale Modeling and
Simulation}, \textbf{4} (2005), 1091--1115.


\bibitem {Korpelevich76}G. M. Korpelevich, The extragradient method for
finding saddle points and other problems, \textit{Ekonomika i Matematicheskie
Metody} \textbf{12} (1976), 747--756.


\bibitem{Lefkimmiatis2015}
S.~Lefkimmiatis, A.~Roussos, P.~Maragos, and M.~Unser,
\newblock Structure tensor total variation,
\newblock {\em SIAM Journal on Imaging Sciences} \textbf{8} (2015), 1090--1122.

\bibitem{Lenzen2015}
F.~Lenzen and J.~Berger,
\newblock Solution-driven adaptive total variation regularization,
\newblock In J.-F. Aujol, M.~Nikolova, and N.~Papadakis, editors, {\em
  Proceedings of SSVM 2015}, volume 9087 of {\em LNCS}, pages 203--215, 2015.


\bibitem {lhz}Z. Li, D. Han and W. Zhang, A self-adaptive projection-type
method for nonlinear multiple-sets split feasibility problem, \textit{Inverse
Problems in Science and Engineering} \textbf{21} (2013), 155--170.

\bibitem {lmwx12}G. Lopez, V. Martin-Marquez, F. Wang and H.-K Xu, Solving the
split feasibility problem without prior knowledge of matrix norms,
\textit{Inverse Problems} \textbf{28:085004} (2012).

\bibitem {mr07}E. Masad and S. Reich, A note on the multiple-set split convex
feasibility problem in Hilbert space, \textit{Journal of Nonlinear and Convex
Analysis} \textbf{8} (2007), 367--371.

\bibitem {popa12}C. Popa, \textit{Projection Algorithms - Classical Results
and Developments: Applications to Image Reconstruction}, Lambert Academic
Publishing - AV Akademikerverlag GmbH \& Co. KG, Saarbr{\"{u}}cken, Germany, 2012.

\bibitem{Rudin1992}
L.~Rudin, S.~Osher, and E.~Fatemi,
\newblock Nonlinear total variation based noise removal
algorithms, \newblock {\em Physica D} \textbf{60} (1992),
259--268.

\bibitem{Scherzer2009variational}
O.~Scherzer, M.~Grasmair, H.~Grossauer, M.~Haltmeier, and
F.~Lenzen, \newblock {\em Variational Methods in
Imaging},\newblock Springer, 2009.


\bibitem{Setzer2011}
S.~Setzer, G.~Steidl, and T.~Teuber,
\newblock Infimal convolution regularizations with discrete l1-type
functionals, \newblock {\em Communications in Mathematical
Sciences} \textbf{9} (2011), 797--872.

\bibitem {sci15} Y. Shehu, G. Cai and O. S. Iyiola, Iterative approximation of solutions for
proximal split feasibility problems, \textit{Fixed Point Theory
and Applications}, \textbf{2015:123} (2015).

\bibitem{Wang2004}
Z.~Wang, A.~Bovik, H.~Sheikh, and E.~Simoncelli,
\newblock Image quality assessment: from error visibility to structural
  similarity, \newblock {\em IEEE Transactions on Image Processing} \textbf{13} (2004), 600--612.


\bibitem{Weickert1998}
J.~Weickert, \newblock {\em Anisotropic Diffusion in Image
Processing}, \newblock Teubner Verlag, 1998.

\bibitem {xcmg03}Y. Xiao, Y. Censor, D. Michalski and J.M. Galvin, The
least-intensity feasible solution for aperture-based inverse
planning in radiation therapy, \textit{Annals of Operations
Research} \textbf{119 }(2003), 183--203.
\end{thebibliography}

\end{document}